\documentclass[12pt]{amsart}
\usepackage{amssymb}
\usepackage{verbatim}
\usepackage{hyperref}
\usepackage{color}

\begin{document}
\setcounter{tocdepth}{1}

\newtheorem{theorem}{Theorem}    
\newtheorem{proposition}[theorem]{Proposition}
\newtheorem{conjecture}[theorem]{Conjecture}
\def\theconjecture{\unskip}
\newtheorem{corollary}[theorem]{Corollary}
\newtheorem{lemma}[theorem]{Lemma}
\newtheorem{sublemma}[theorem]{Sublemma}
\newtheorem{fact}[theorem]{Fact}
\newtheorem{observation}[theorem]{Observation}
\theoremstyle{definition}
\newtheorem{definition}{Definition}
\newtheorem{notation}[definition]{Notation}
\newtheorem{remark}[definition]{Remark}
\newtheorem{question}[definition]{Question}
\newtheorem{questions}[definition]{Questions}

\newtheorem{example}[definition]{Example}
\newtheorem{problem}[definition]{Problem}
\newtheorem{exercise}[definition]{Exercise}

\numberwithin{theorem}{section}
\numberwithin{definition}{section}
\numberwithin{equation}{section}

\def\reals{{\mathbb R}}
\def\torus{{\mathbb T}}
\def\integers{{\mathbb Z}}
\def\rationals{{\mathbb Q}}
\def\naturals{{\mathbb N}}
\def\complex{{\mathbb C}\/}
\def\heis{{\mathbb H}\/}
\def\distance{\operatorname{distance}\,}
\def\diststar{{\rm dist}^\star}
\def\sym{\operatorname{Symm}\,}
\def\support{\operatorname{support}\,}
\def\dist{\operatorname{dist}}
\def\Span{\operatorname{span}\,}
\def\degree{\operatorname{degree}\,}
\def\kernel{\operatorname{kernel}\,}
\def\dim{\operatorname{dim}\,}
\def\codim{\operatorname{codim}}
\def\trace{\operatorname{trace\,}}
\def\Span{\operatorname{span}\,}
\def\dimension{\operatorname{dimension}\,}
\def\codimension{\operatorname{codimension}\,}
\def\Gl{\operatorname{Gl}}
\def\nullspace{\scriptk}
\def\kernel{\operatorname{Ker}}
\def\ZZ{ {\mathbb Z} }
\def\p{\partial}
\def\rp{{ ^{-1} }}
\def\Re{\operatorname{Re} }
\def\Im{\operatorname{Im} }
\def\ov{\overline}
\def\eps{\varepsilon}
\def\lt{L^2}
\def\diver{\operatorname{div}}
\def\curl{\operatorname{curl}}
\def\etta{\eta}
\newcommand{\norm}[1]{ \|  #1 \|}
\def\expect{\mathbb E}
\def\bull{$\bullet$\ }
\def\det{\operatorname{det}}
\def\Det{\operatorname{Det}}
\def\multiR{\mathbf R}
\def\bestA{\mathbf A}
\def\bestB{\mathbf B}
\def\bestC{\mathbf C}
\def\Apq{\mathbf A_{p,q}}
\def\Apqr{\mathbf A_{p,q,r}}
\def\rank{\operatorname{rank}}
\def\rankk{\mathbf r}
\def\diameter{\operatorname{diameter}}
\def\bp{\mathbf p}
\def\bff{\mathbf f}
\def\bg{\mathbf g}
\def\essd{\operatorname{essential\ diameter}}

\def\mab{M}
\def\t2{\tfrac12}

\newcommand{\abr}[1]{ \langle  #1 \rangle}
\def\unitQ{{\mathbf Q}}
\def\mbfp{{\mathbf P}}

\def\aff{\operatorname{Aff}}
\def\T{{\mathcal T}}

\newcommand{\Norm}[1]{ \Big\|  #1 \Big\| }
\newcommand{\set}[1]{ \left\{ #1 \right\} }
\newcommand{\sset}[1]{ \{ #1 \} }

\def\one{{\mathbf 1}}
\def\onei{{\mathbf 1}_I}
\def\onee{{\mathbf 1}_E}
\def\onea{{\mathbf 1}_A}
\def\oneb{{\mathbf 1}_B}
\def\wonee{\widehat{\mathbf 1}_E}
\newcommand{\modulo}[2]{[#1]_{#2}}

\def\rint{ \int_{\reals^+} }
\def\Abest{{\mathbb A}}
\def\op{\operatorname{Op}}
\def\and{ \ \text{ and }\ }

\def\barrier{\medskip\hrule\hrule\medskip}

\def\symdif{\,\Delta\,}
\def\akd{{\mathbf A}_{k,d}}
\def\ak{{\mathbf A}_{k}}

\def\scriptf{{\mathcal F}}
\def\scripts{{\mathcal S}}
\def\scriptq{{\mathcal Q}}
\def\scriptg{{\mathcal G}}
\def\scriptm{{\mathcal M}}
\def\scriptb{{\mathcal B}}
\def\scriptc{{\mathcal C}}
\def\scriptt{{\mathcal T}}
\def\scripti{{\mathcal I}}
\def\scripte{{\mathcal E}}
\def\scriptv{{\mathcal V}}
\def\scriptw{{\mathcal W}}
\def\scriptu{{\mathcal U}}
\def\scripta{{\mathcal A}}
\def\scriptr{{\mathcal R}}
\def\scripto{{\mathcal O}}
\def\scripth{{\mathcal H}}
\def\scriptd{{\mathcal D}}
\def\scriptl{{\mathcal L}}
\def\scriptn{{\mathcal N}}
\def\scriptp{{\mathcal P}}
\def\scriptk{{\mathcal K}}
\def\scriptP{{\mathcal P}}
\def\scriptj{{\mathcal J}}
\def\scriptz{{\mathcal Z}}
\def\frakv{{\mathfrak V}}
\def\frakG{{\mathfrak G}}
\def\frakA{{\mathfrak A}}
\def\frakB{{\mathfrak B}}
\def\frakC{{\mathfrak C}}
\def\frakf{{\mathfrak F}}
\def\fraki{{\mathfrak I}}
\def\fcross{{\mathfrak F^{\times}}}

\def\boldf{\mathbf f}
\def\boldx{\mathbf x}
\def\boldw{\mathbf w}
\def\bolda{\mathbf a}
\def\boldb{\mathbf b}
\def\boldg{\mathbf g}
\def\bF{\mathbf F}
\def\bE{\mathbf E}
\def\be{\mathbf e}
\def\bA{\mathbf A}
\def\ba{\mathbf a}
\def\bEstar{\mathbf E^\star}
\def\bAstar{\mathbf A^\star}
\def\s01{ \{0,1\} }
\def\sok{ \{0,1\}^k }

\author{Michael Christ}
\address{
        Michael Christ\\
        Department of Mathematics\\
        University of California \\
        Berkeley, CA 94720-3840, USA}
\email{mchrist@math.berkeley.edu}
\thanks{Research supported in part by NSF grant DMS-1363324.}

\date{November 18, 2015.} 

\title[Nearly maximal Gowers norms]
{Subsets of Euclidean space \\ with nearly maximal Gowers norms}

\begin{abstract} A set $E\subset\reals^d$ whose indicator function $\one_E$
has maximal Gowers norm, among all sets of equal measure, is an ellipsoid up to Lebesgue null sets.
If $\one_E$ has nearly maximal Gowers norm then $E$ nearly coincides with an ellipsoid.
\end{abstract}

\maketitle

\section{Introduction}
Let $d\ge 1$.
Let $\norm{f}_{U_k}$ be the Gowers norm of order $k\ge 2$
of a Lebesgue measurable function $f:\reals^d\to\complex$.
These norms\footnote{The $U_1$ ``norm'' is not actually a norm, but provides
a convenient base for the inductive definition.}
are defined inductively by $\norm{f}_{U_1}=|\int f|^2$ and
\[ \norm{f}_{U_{k+1}}^{2^{k+1}}
= \int_{\reals^d} \norm{f\,\overline{f^s}}_{U_k}^{2^k}\,ds\]
where $f^s(x)=f(x+s)$,
provided that the integral converges absolutely. 
Integration is with respect to Lebesgue measure.
See \cite{taovu} for basic properties of these functionals.

In this note we are primarily concerned with indicator functions,
those of the form $f(x)=\one_E(x)=1$ if $x\in E$ and $=0$ if $x\notin E$.
We systematically abuse notation by writing $\norm{E}_{U_k}$ as shorthand for $\norm{\one_E}_{U_k}$.
Thus $\norm{E}_{U_1}^2 = |E|^2$ and
\begin{equation} \norm{E}_{U_{k+1}}^{2^{k+1}} 
= \int_{\reals^d} \norm{{E\cap(E+s)}}_{U_k}^{2^k}\,ds \end{equation}
for $k\ge 1$, where $E+s=\{x+s: x\in E\}$.
In certain applications, these norms have arisen for functions whose
domains are other groups such as $\integers$ or finite cyclic groups, and a key condition has
been that a norm is not arbitrarily small. In the present paper, the focus
is on the Euclidean group $\reals^d$ and on sets whose norm is 
as large as possible, or nearly so.

Denote by $A\symdif B$ the symmetric difference $(A\cup B)\setminus (A\cap B)$ of two
sets, and by $|E|$ the Lebesgue measure of a set $E$.
Of course, $\norm{A}_{U_k}=\norm{B}_{U_k}$ whenever $|A\symdif B|=0$.
All sets mentioned are assumed to be Lebesgue measurable, even when this is not explicitly stated.

The normalized ratio
$\norm{E}_{U_k}\,/\,{|E|^{(k+1)/2^k}}$
is affine-invariant; if $\phi:\reals^d\to\reals^d$ is an affine automorphism then
\begin{equation} \label{eq:affineinvariant} 
\frac{\norm{{\phi(E)}}_{U_k}}{|\phi(E)|^{(k+1)/2^k}} 
= \frac{\norm{{E}}_{U_k}}{|E|^{(k+1)/2^k}}. 
\end{equation}
Therefore
there exist constants $\gamma_{k,d}$ satisfying
\begin{equation} \label{gammadefn}
\norm{\scripte}_{U_k}^{2^k} = \gamma_{k,d} |\scripte|^{k+1} \text{ for every ellipsoid $\scripte\subset\reals^d$.}
\end{equation}

For any Lebesgue measurable set $E\subset\reals^d$ with $|E|<\infty$,
denote by $E^\star\subset\reals^d$ the closed ball centered at $0$ that satisfies $|E^\star|=|E|$.
For $f:\reals^d\to[0,\infty]$ 
denote by 
$f^\star:\reals^d\to[0,\infty]$
its radially symmetric function nonincreasing symmetrization, 
whose definition is reviewed below. 

According to a more general theorem of Brascamp, Lieb, and Luttinger \cite{BLL}, 
$\norm{f}_{U_k}\le \norm{f^\star}_{U_k}$ for arbitrary functions $f$. In particular,
\begin{equation} \norm{E}_{U_k} \le  \norm{{E^\star}}_{U_k} \end{equation}
for all measurable sets $E$ with finite measures.
Thus \begin{equation}\label{eq:powerlaw} 
\norm{E}_{U_k}^{2^k} \le \gamma_{k,d}|E|^{k+1}\end{equation} 
for arbitrary sets, 
and by the affine invariance \eqref{eq:affineinvariant},  
all ellipsoids are among the maximizers of the ratio $\norm{E}_{U_k}\,/\,|E|^{(k+1)/2^k}$.

Our first result states that up to modification by Lebesgue null sets, there are no other maximizing sets.
\begin{theorem} \label{thm:1}
Let  $k\ge 2$ and $d\ge 1$.
Let $E\subset\reals^d$ be a Lebesgue measurable set with $0<|E|<\infty$.
Then $\norm{E}_{U_k}=\norm{{E^\star}}_{U_k}$ if and only if
there exists an ellipsoid $\scripte\subset\reals^d$ such that $|E\symdif \scripte|=0$.
\end{theorem}

For general functions a related result was previously known.
Set $p_k = 2^k(k+1)^{-1}$. Then 
$\norm{f}_{U_k} \le A_{k,d}\norm{f}_{p_k}$ for all $f\in L^{p_k}(\reals^d)$
for a certain constant $A_{k,d}<\infty$.
The optimal constant $A_{k,d}$ in this inequality has been calculated explicitly by Eisner and Tao \cite{eisnertao},
who also showed that only Gaussian functions $f$ maximize $\norm{f}_{U_k}\,/\,\norm{f}_{p_k}$.
An easy consequence is that an arbitrary complex-valued function is a maximizer if and only if
it takes the form $f = Ge^{i\phi}$ where $G$ is a real Gaussian
and $\phi:\reals^d\to\reals$ is a real-valued polynomial of degree strictly less than $k$.
However, this leaves open the question of characterizing functions
satisfying $\norm{f}_{U_k}=\norm{f^\star}_{U_k}$ 
without maximizing the ratio $\norm{f^\star}_{U_k}/\norm{f}_{p_k}$. 

By a radial nonincreasing function is meant $f:\reals^d\to[0,\infty)$
of the form $f(x)=g(|x|)$ Lebesgue almost everywhere, 
where $g$ has domain $(0,\infty)$ and $g$ is nonincreasing.
Only the ``only if'' implication in the following statement is nontrivial.
\begin{corollary} \label{cor:functions}
Let $d\ge 1$ and $k\ge 2$. 
Let $f:\reals^d\to[0,\infty]$ be Lebesgue measurable and suppose
that $\norm{f^\star}_{U_k}<\infty$.
Then $\norm{f}_{U_k}=\norm{f^\star}_{U_k}$
if and only if
there exists an affine automorphism $\phi$ of $\reals^d$
such that $f\circ\phi$ agrees almost everywhere with a radial nonincreasing function.
\end{corollary}

Our third result is a more robust form of Theorem~\ref{thm:1}.
\begin{theorem} \label{thm:2}
Let $k\ge 2$ and $d\ge 1$.
For any $\eps>0$ there exists $\delta>0$ such that for any Lebesgue measurable set
$E\subset\reals^d$ satisfying
$0<|E|<\infty$ and
\begin{equation} \norm{E}_{U_k} \ge (1-\delta) \norm{{E^\star}}_{U_k}, \end{equation}
there exists an ellipsoid $\scripte$ such that
\begin{equation} |E\symdif \scripte| <\eps |E|.  \end{equation}
\end{theorem}

Results like Theorems~\ref{thm:1} and \ref{thm:2},
concerning indicator functions of sets,
are more fundamental than analogous results for general nonnegative functions,
and can be applied to general functions by representing the latter 
as superpositions of indicator functions \cite{liebloss}.
See for instance the proof of Corollary~\ref{cor:functions} given below,
and \cite{christyoungest}, for two applications in this spirit.

While Theorem~\ref{thm:2} is the main goal of this paper,
the author is not aware of the more elementary 
Theorem~\ref{thm:1} having previously been noted in the literature.
\S\ref{section:prelim} introduces basic elements of the analysis.
\S\ref{section:chain} develops an alternative proof of the general inequality
$\norm{E}_{U_k}\le\norm{{E^\star}}_{U_k}$,
relying on the classical Riesz-Sobolev inequality. This proof yields additional
information not obtained through a direct application of the Brascamp-Lieb-Luttinger inequality.
In \S\ref{section:equality} this additional information is exploited 
to characterize cases of equality. In \S\ref{section:near} we combine the argument
developed in earlier sections with the Riesz-Sobolev stability theorem of \cite{christRS3} 
to prove Theorem~\ref{thm:2}.
\S\ref{section:corollary} contains the proof of Corollary~\ref{cor:functions}. 

A natural goal is to extend these results to multilinear forms which underlie
the Gowers norms. Thus one has certain functionals of $2^k$--tuples
${\mathbf E} = (E_\alpha: \alpha\in\{0,1\}^k)$ of sets. In the case in which $E_\alpha=E$
for all $\alpha$, one obtains $\norm{E}_{U_k}^{2^k}$. The supremum of this
quantity over all sets
$E$ of some specified Lebesgue measure $m$ takes the simple form $\gamma_{k,d} m^{k+1}$,
and the analysis below exploits this simple algebraic expression.
In contrast, the analogous supremum in the multilinear context is
a more complicated function of $(|E_\alpha|: \alpha\in\{0,1\}^k)$; this function
is a piecewise cubic polynomial in four variables for $k=2$,
and appears to be prohibitively complicated for $k\ge 3$. 
We nonetheless hope to treat these questions in a sequel by a related argument.

The author thanks Anh Nguyen for useful comments on the exposition.

\section{Preliminaries} \label{section:prelim}

Let
$f:\reals^d\to[0,\infty]$ be Lebesgue measurable
and assume that
$|\{x: f(x)>t\}|<\infty$ for all $t>0$.
To such a function we associate auxiliary functions $f^\star$ and $f_*$. 

\begin{definition} \label{defn:1}
The radially symmetric nonincreasing symmetrization $f^\star:\reals^d\to[0,\infty]$
of $f$ is the unique radially symmetric function that satisfies
\[|\{x: f^\star(x)>t\}| =|\{x: f(x)>t\}|\ \ \text{for all $t>0$,} \]
with the right continuity property
\[\lim_{|x|\to |y|^+} f^\star(x)=f^\star(y)\  \  \text{for all $y\in\reals^d$.}\]
\end{definition}

\begin{definition} \label{defn:2}
$f_*:(0,\infty)\to[0,\infty)$ 
is defined to be the unique right-continuous nonincreasing function satisfying
\begin{equation} \label{fstardef}
|\{t\in\reals^+: f_*(t) > \alpha\}| = |\{x\in\reals^d: f(x) > \alpha\}|\ \ 
\text{ for almost every } \alpha>0.\end{equation}
\end{definition}

$f^\star$ has domain $\reals^d$, while $f_*$ has domain $(0,\infty)$.
In the next definition, $E+s=\{x+s: x\in E\}$.
\begin{definition} \label{defn:3}
Let $E\subset\reals^d$ be Lebesgue measurable with $|E|<\infty$. 
The associated functions $f,F,\tilde f,\tilde F$ are defined by
\begin{align} 
f(s)&=|E\cap(E+s)| 
\\ F(t)& =\int_0^t f_*(\tau)\,d\tau \ \ \text{ for $t\in\reals^+$.} 
\\ \tilde f(s) &=|E^\star\cap (E^\star+s)|
\\ \tilde F(t)& =\int_0^t \tilde f_*(\tau)\,d\tau \ \ \text{ for $t\in\reals^+$} 
\end{align} 
where $\tilde f_* = (\tilde f)_*$.
\end{definition}
The functions $f_*,\tilde f_* = (\tilde f)_*$ mapping $\reals^+$ to $[0,\infty)$ 
are defined in terms of $f,\tilde f$ by Definition~\ref{defn:2}.
The functions $f,F,\tilde f,\tilde F$ should perhaps be denoted as $f_E,F_E,\tilde f_E,\tilde F_E$ respectively,
but only one set $E$ will be under discussion at a time so the simplified notation will be used.

An equivalent description of $F$  will be used below.
\begin{lemma} \label{lemma:bathtub}
Let $f:\reals^d\to[0,\infty]$ be a Lebesgue measurable function
such that $|\{x: f(x)>t\}|<\infty$ for all $t>0$.
Let $F(t) = \int_0^t f_*(s)\,ds$.
For any $t>0$,
\begin{equation} \label{describeF} F(t) = \sup_{A:\,\, |A|=t}\ \int_A f\end{equation}
where the supremum is taken over all Lebesgue measurable sets $A\subset\reals^d$
satisfying $|A|=t$.
\end{lemma}
The equivalence is straighforward and well known, so we omit the proof.
See for instance Theorem~1.14 of \cite{liebloss}.
By applying this to $E^\star$, we obtain a corresponding description of $\tilde F$.

\begin{lemma} \label{lemma:FFtilde}
Let $E\subset\reals^d$ be a Lebesgue measurable set with $|E|<\infty$.
Let $f,\tilde f,F,\tilde F$ be associated to $E$ as above.
Then
\begin{equation}\label{ineq:FFtilde} F(s)\le \tilde F(s)\ \text{ for all $s\in\reals^+$.} \end{equation}
Moreover, if there exists $s\in (0,2^d |E|)$ for which $F(s)=\tilde F(s)$
then there exists an ellipsoid $\scripte\subset\reals^d$ satisfying $|E\symdif\scripte|=0$.
\end{lemma}

The proof relies on a theorem of Burchard \cite{burchard}. 
If $(|A|,|B|,|C|)$ is strictly admissible in the sense that the sum of
any two of these quantities is strictly greater than the third,
and if $\langle \one_A*\one_B,\one_C\rangle =\langle \one_{A^\star}*\one_{B^\star},\one_{C^\star}\rangle$,
then there exists an ellipsoid $\scripte$
such that $|A\symdif \scripte|=0$.

\begin{proof}[Proof of Lemma~\ref{lemma:FFtilde}]
The Riesz-Sobolev inequality states that 
for any three measurable sets $A,B,C\subset\reals^d$,
$\int_A \one_B*\one_C \le \int_{A^\star} \one_{B^\star}*\one_{C^\star}$.
Writing \[\one_B*\one_{-C}(x) = \int \one_B(y)\one_{-C}(x-y)\,dy
= \int \one_B(y)\one_{C}(y-x)\,dy\] where $-C=\{-b: b\in C\}$ gives
$\int_A \one_B*\one_C = \int \one_A(x) |B\cap (\tilde C-x)|\,dx$. 

Upon putting $f(s)=|E\cap(E+s)|$ one obtains
\begin{equation} \sup_{A:\,\, |A|=r}\ \int_A f\le \sup_{A:\,\, |A|=r}\ \int_A \tilde f\end{equation}
for all $r\in\reals^+$. This can be rewritten in terms of $F,\tilde F$ as
\begin{equation} F(s)\le\tilde F(s) \ \text{for all $s\in(0,\infty)$.} \end{equation} 

Now suppose that $s\in(0,2^d|E|)$ and $F(s)=\tilde F(s)$.
By Lemma~\ref{lemma:bathtub} there exists a set $A\subset\reals^d$
satisfying $|A|=s$ such that $\int_{A} |E\cap(E+y)|\,dy
= F(s)$. Likewise, $\tilde F(s)=\int_{A^\star} |E^\star\cap(E^\star+y)|\,dy$.
The condition $F(s)=\tilde F(s)$ thus means that
\begin{equation}\label{applyburchard} \int_{A} |E\cap(E+y)|\,dy = \int_{A^\star} |E^\star\cap(E^\star+y)|\,dy.
\end{equation}
The ordered triple $(|E|^{1/d},|E|^{1/d},|A|^{1/d})$ is strictly admissible when $0<|A|<2^d|E|$;
the sum of any two of these quantities is strictly greater than the third.
Therefore by Burchard's theorem, $E$ is an ellipsoid, up to Lebesgue null sets.
\end{proof}

The relation $F(s)=\tilde F(s)$ holds for some $s\in (0,2^d|E|)$ 
if and only if it holds for all $s$ in this interval.
For the latter implies that $E$ is an ellipsoid up to null sets;
by affine invariance, $F(s),\tilde F(s)$ are unchanged if this ellipsoid is replaced by 
a ball centered at the origin of equal measure;
and $F$ is identically equal to $\tilde F$ when $E$ is a ball.

\section{A chain of inequalities}\label{section:chain}

Let $\gamma_{k,d}$ be as defined in \eqref{gammadefn}.
\begin{lemma} \label{lemma:chaink}
Let $k\ge 2$ and $d\ge 1$.
For any Lebesgue measurable set $E\subset\reals^d$ satisfying $|E|<\infty$,
\begin{equation} \label{chain}
\norm{E}_{U_k}^{2^k}
\le \gamma_{k-1,d} \int_{\reals^+} f_* \tilde f_*^{k-1}
\le \gamma_{k-1,d} \int_{\reals^+} \tilde f_*^{k}
= \norm{{E^\star}}_{U_k}^{2^k}.
\end{equation} 
\end{lemma}

\begin{proof}
The proof is by induction on $k$, with $k=2$ as the base case.
We begin with the inductive step.
Define $\mu$ to be the unique nonnegative measure on $\reals^+$ 
that satisfies \[\int_{\reals^+} \varphi f_*^{k-1}  = \int_{\reals^+} \Phi\,d\mu\]
for any continuous compactly supported function $\varphi$,
where $\Phi(x)=\int_0^x \varphi(t)\,dt$.
That is, $-\mu$ is the derivative of $f_*^{k-1}$ on $(0,\infty)$, in the sense of distributions.

All integrals in the remainder of the proof of Lemma~\ref{lemma:chaink} are taken over $\reals^+$.
Integration by parts followed by the pointwise inequality $F\le \tilde F$ gives
\[ \int f_*^{k} = \int F\,d\mu \le \int \tilde F\,d\mu.\]
Integrating by parts in reverse gives
$\int \tilde F\,d\mu = \int \tilde f_*\cdot f_*^{k-1}$, so 
\[ \int f_*^{k} \le \int \tilde f_*\cdot f_*^{k-1}.\]

In this last integrand, factor $\tilde f_*\cdot f_*^{k-1} = f_*\cdot (\tilde f_*\cdot f_*^{k-2})$.
Repeat the reasoning of the preceding paragraph, 
but with $-\mu$ now equal to the derivative of the nonincreasing function $\tilde f_* f_*^{k-2}$,
to conclude that
$\int \tilde f_*\cdot f_*^{k-1} \le \int \tilde f_*^2\cdot f_*^{k-2}$
and hence that
$\int f_*^{k} \le \int \tilde f_*^2\cdot f_*^{k-2}$.

This reasoning can be iterated. Each iteration converts one factor of $f_*$ to 
a factor of $\tilde f_*$, to yield a chain of inequalities:
\begin{equation} 
\int_{\reals^+} f_*^{k} 
\le \int_{\reals^+} f_*^{k-1} \tilde f_*
\le \int_{\reals^+} f_*^{k-2} \tilde f_*^2
\le \cdots
\le \int_{\reals^+} f_* \tilde f_*^{k-1}
\le \int_{\reals^+} \tilde f_*^{k}.
\end{equation} 

By the inductive hypothesis,
\begin{equation} \label{applyinduction} \norm{{E\cap(E+s)}}_{U_{k-1}}^{2^{k-1}} 
\le \norm{{[E\cap(E+s)]^\star}}_{U_{k-1}}^{2^{k-1}} .
\end{equation}
Moreover,
\[ \norm{{[E\cap(E+s)]^\star}}_{U_{k-1}}^{2^{k-1}} 
= \gamma_{k-1,d}|[E\cap(E+s)]^\star|^{k} 
= \gamma_{k-1,d}|E\cap(E+s)|^{k} 
= \gamma_{k-1,d}f(s)^{k};\]
the final identity is an application of \eqref{eq:powerlaw}.
Because $f_*$ and $f$ are equimeasurable,
$\int f_*^{k}=\int f^{k}$. 
Therefore
\begin{equation}
\norm{E}_{U_k}^{2^k}
= \int \norm{{E\cap(E+s)}}_{U_{k-1}}^{2^{k-1}}\,ds
\le \gamma_{k-1,d} 
\int f_*^{k}(s)\,ds.
\end{equation}
In the same way,
\[ \gamma_{k-1,d} \int \tilde f_*^{k}\,ds 
= \gamma_{k-1,d} \int \tilde f^{k}\,ds 
= \gamma_{k-1,d} \int |[E^\star\cap(E+s)^\star]|^{k}\,ds = \norm{{E^\star}}_{U_k}^{2^k}.  \]
This completes the induction step.

For the base case $k=2$, we repeat the same argument step by step.
The inequality \eqref{applyinduction} is no longer justified by
the inductive hypothesis, but instead, it is a trivial equality.
Indeed, for any set $E$ with finite Lebesgue measure,
$\norm{E}_{U_1}^2 = |E|^2 = |E^\star|$.
Therefore $\gamma_{1,d}=1$, yielding \eqref{applyinduction} with $k-1=1$.
All of the other steps of the argument are unchanged.
\end{proof}

\section{Characterization of equality}\label{section:equality}

\begin{proof}[Proof of Theorem~\ref{thm:1}]
Let $E$ be a set with $0<|E|<\infty$ that satisfies
$\norm{E}_{U_k} = \norm{{E^\star}}_{U_k}$.
Let $f,f_*,F, \tilde f,\tilde f_*,\tilde F$ be defined in terms of $E$ as above.
In order to show that $E$ is an ellipse up to sets of measure zero,
it suffices by Burchard's theorem \cite{burchard} 
to show that $F(s)=\tilde F(s)$ for some $s\in(0,2^d|E|)$;
for this reduction see Lemma~\ref{lemma:FFtilde}.

To prove the existence of such an $s$, recall that 
the inequalities in Lemma~\ref{lemma:chaink}
necessarily become equalities when $\norm{E}_{U_k}=\norm{E^\star}_{U_k}$. In particular,
$\int f_* \tilde f_*^{k-1} = \int \tilde f_*^{k}$. 
We have already seen that this can equivalently be rewritten 
\begin{equation}\label{FFtilde} \int F\,d\mu = \int \tilde F\,d\mu \end{equation}
where the nonnegative measure $\mu$ is the derivative of the nondecreasing
function $-\tilde f_*^{k-1}$ in the sense of distributions.

According to Lemma~\ref{lemma:FFtilde}, $F\le \tilde F$ pointwise. Thus
\eqref{FFtilde} forces $F=\tilde F$ almost everywhere with respect to $\mu$.
The function $\tilde f_*$ is defined by the relation
\[ \big|\{s\in\reals^+: \tilde f_*(s)>\alpha\}\big| = \big|\{x\in\reals^d: |E^\star\cap (E^\star+x)|>\alpha\}\big|
\qquad \forall\,\,\alpha>0\]
where $E^\star=B$ is the closed ball centered at $0$ of measure $|E|$.
By examining $|B\cap(B+y)|$ 
one sees that $\mu$ and Lebesgue measure are mutually absolutely continuous 
on the interval $(0,2^d|E|)$. Of course, $F,\tilde F$ are continuous functions.
Therefore 
$F(s)=\tilde F(s)$ for every $s\in(0,2^d|E|)$.
\end{proof}

\section{Near equality}\label{section:near}

\begin{lemma} \label{lemma:nuisance}
For any compact subinterval $J$ of $(0,2^d)$ of positive length
there exists $C<\infty$, depending only on $J,k,d$, with the following property.
Let $\delta\in(0,1]$ be arbitrary.  Let $E\subset\reals^d$ satisfy
$\norm{E}_{U_k}^{2^k} \ge (1-\delta)\norm{E^\star}_{U_k}^{2^k}$ and $|E|=1$. 
Let $F,\tilde F,\mu$ be as in Definitions~\ref{defn:2} and \ref{defn:3}.  Then
\begin{equation} \label{ineq:oftenclose}
\mu\big(\{s\in J: F(s)\le (1-\delta^{1/2})\tilde F(s)\}\big) \le C\delta^{1/2}.
\end{equation}
\end{lemma}

\begin{proof}
The function $\tilde F$ and measure $\mu$ depend only on the parameters $d,k$,
not on the set $E$. 
Indeed, for fixed parameters $k,d$, both $\tilde F$ and $\mu$ are defined in terms of $E^\star$,
thus in terms of $|E|$ alone, and we have assumed that $|E|=1$. 
Thus $E^\star$ is the ball $B$ centered
at the origin whose Lebesgue measure equals $1$, and 
$\norm{E^\star}_{U_k}$ is a finite quantity, which
depends only on $k$ and on the dimension of the ambient space.
Likewise, $\tilde f (s) = |B\cap(B+s)|$
is a continuous nonnegative function which is independent of $E$,    
and which is strictly positive on the open ball $B'$ centered at the origin
whose radius is equal to twice the radius of $B$.
This function is radially symmetric, is nonincreasing along rays emanating from
the origin, and has strictly negative radial derivative at each nonzero point
in the interior of $B'$.
Therefore $\tilde f_*$ is a continuous function which is supported
on $[0,2^d]$, is strictly positive on $[0,2^d)$, is nonincreasing, and has strictly
negative derivative at each point of the open interval $(0,2^d)$.
$\mu$ is the absolutely continuous measure whose Radon-Nikodym derivative equals
 $-(k-1)\tilde f_*(t)^{k-2}\frac{d}{dt}\tilde f_*(t)$.
This is a finite measure, since $\tilde f_*(0) = |B\cap(B+0)|=|B|^2=1<\infty$.
$\mu$ and Lebesgue measure are mutually comparable on any compact subinterval of $(0,2^d)$.
The indefinite integral $\tilde F$ of $\tilde f_*$ is bounded, so $\int_{\reals^+} \tilde F\,d\mu$
is finite.

By choosing $C$ large we may assume that $\delta$ is small, since the conclusion holds
if $C\delta>\mu(J)$.
By the reasoning used to deduce \eqref{FFtilde} in the analysis of cases of equality above,
$\int_0^\infty F\,d\mu \ge (1-\delta) \int_0^\infty \tilde F\,d\mu$.
A consequence is that
\begin{multline*}
\int_J F\,d\mu
 = \int_{\reals^+} F\,d\mu - \int_{\reals^+\setminus J} F\,d\mu
\ge (1-\delta)\int_{\reals^+} \tilde F\,d\mu - \int_{\reals^+\setminus J} \tilde F\,d\mu
\\
 = \int_J \tilde F\,d\mu -\delta\int_{\reals^+} \tilde F\,d\mu
= \int_J \tilde F\,d\mu - \delta \norm{E^\star}_{U_k}^{2^k}.
\end{multline*}
Therefore $\int_J (\tilde F-F)\,d\mu  \le \delta\norm{E^\star}_{U_k}^{2^k}$.  

Let $\eta>0$.
Recalling that $\tilde F-F\ge 0$ 
invoking Chebyshev's inequality gives
$\tilde F(s)\le F(s) + \eta$ for all $s\in J\setminus A$ 
where the exceptional set $A$ satisfies $\mu(A)\le \eta^{-1}\delta\norm{E^\star}_{U_k}^{2^k}$.
Choose $\eta = c\delta^{1/2}$ where $c$ is a small constant.
Since $\tilde F$ is bounded below on $J$ by a positive quantity which depends only on $d,k,J$,
$F(s)\ge (1-\delta^{1/2})\tilde F(s)$ for every $s\in J\setminus A$, 
provided that $c$ is chosen to be sufficiently small.
\end{proof}

\begin{proof}[Proof of Theorem~\ref{thm:2}]
Let $d,k$ be given, let $\delta>0$ be small, and 
let $E\subset\reals^d$ be any Lebesgue measurable set of finite measure
satisfying $\norm{E}_{U_k}^{2^k} \ge (1-\delta)\norm{{E^\star}}_{U_k}^{2^k}$.
By dilation invariance of the hypothesis and conclusion, 
we may assume without loss of generality that $|E|=1$.
Choose a compact subinterval $J\subset(0,2^d)$ of positive length.
If $\delta$ is sufficiently small then by Lemma~\ref{lemma:nuisance} and because
$\mu(J)>0$, there exists $s\in J$ for which $F(s)\ge (1-\delta^{1/2}) \tilde F(s)$.
By the characterization \cite{christRS3},\cite{christRShigher} of cases of 
near equality in the Riesz-Sobolev inequality,
existence of even a single such value of $s$
implies the existence of an ellipsoid $\scripte$ satisfying $|E\symdif \scripte|\le \eps|E|$ 
where $\eps=\eps(\delta,d)$ satisfies $\lim_{\delta\to 0^+} \eps(\delta,d)=0$ for each dimension $d$.
\end{proof}

\begin{remark}
It should be possible to analyze the inequality
\begin{equation} \label{ineq:rs} \langle \one_{E_1}*\one_{E_2},\,\one_{E_3}*\one_{E_4}\rangle
\ge (1-\delta) \langle \one_{E_{1}^\star}*\one_{E_2^\star},\,\one_{E_3^\star}*\one_{E_4^\star}\rangle \end{equation}
for small $\delta>0$
with four arbitrary subsets $E_j\subset \reals^d$ with finite Lebesgue measures,
under the appropriate strict admissibility hypothesis, by this same method. 
The left-hand side of \eqref{ineq:rs} equals
\[ \int_{\reals^d} |E_1\cap(E_2+s)|\cdot|E_3\cap(E_4+s)|\,ds,\]
which is amenable to the above analysis since the integrand is a product.
The multilinear forms corresponding to greater values of $k$ are seemingly more difficult to analyze.
\end{remark}

\section{Digression on measures of slices of convex sets}

In the proof of Corollary~\ref{cor:functions}
we will use certain properties of convex sets.
We pause here to review the required facts.
Let $K\subset\reals^m_y\times \reals^n_x$ be a compact convex set
with positive Lebesgue measure.
For each $y\in\reals^m$ consider the slice
\begin{align*} K_y&=\{x\in \reals^n: (y,x)\in K\}.
\end{align*}
Let $|K_y|$ denote the $n$--dimensional Lebesgue measure of $K_y$. 
Recall that $K$ is said to be balanced if $-K=K$.

\begin{lemma} \label{lemma:convexslices}
Let $K\subset \reals^m_y\times\reals^n_x$ be 
compact, convex, balanced, and have strictly positive, finite Lebesgue measure.
If there exists $0\ne y\in\reals^m$ for which $|K_y|=|K_0|$
then for each $z\in\reals^m$ in the closed segment with endpoints
$0,y$ there exists $w\in\reals^n$ such that
\begin{equation} K_z = K_0+w. \end{equation}
\end{lemma}

\begin{proof}
Consider the function $\varphi(t)= |K_{tv}|$. 
Then $\varphi$ is log concave, that is, 
$\varphi(st+(1-s)t') \ge \varphi(t)^s\varphi(t')^{1-s}$
for all $t,t'\in\reals$ and $s\in(0,1)$.
This is a consequence of the Brunn-Minkowski inequality for $\reals^{n}$,
since whenever $\tau = st+(1-s)t'$ and $K_{tv},K_{t'v}$ are nonempty 
there is the relation $K_{\tau v} \supset sK_{tv}+(1-s)K_{t'v}$.

Since $K$ is balanced, $\varphi(-t)\equiv \varphi(t)$, and log concavity consequently forces 
$\varphi$ to be nonincreasing on $(0,\infty)$.
If $y\ne 0$ satisfies $|K_y|=|K_0|$ then $\varphi(1)=\varphi(0)$ and consequently
$\varphi(t)=\varphi(0)$ for all $t\in[0,1]$.

In particular, $|(1-t)K_0+tK_y| = |K_0|^{1-t}|K_y|^{t}$.
By the characterization of equality in the Brunn-Minkowski inequality, 
$K_{y}$ must be a translate of $K_0$. 
\end{proof}

\section{Proof of the corollary} \label{section:corollary}

\begin{proof}[Proof of Corollary~\ref{cor:functions}]
Let 
$k\ge 2$ and set $p_k = 2^k(k+1)^{-1}$.
There exists a complex-valued $2^k$--linear form ${\mathbf f}=(f_\alpha: \alpha\in\{0,1\}^k)
\mapsto \scriptt_k({\mathbf f})$,
defined for all vector-valued functions ${\mathbf f}$ with components
$f_\alpha\in L^{p_k}(\reals^d)$, satisfying 
\begin{equation}  \norm{f}_{U_k}^{2^k} = \scriptt_k({\mathbf f})
\end{equation}
with ${\mathbf f}=(f_\alpha: \alpha\in\{0,1\}^k)$ where $f_\alpha=f$ for all indices $\alpha$.
See \cite{taovu}.
This form is one to which the symmetrization inequality of Brascamp-Lieb-Luttinger \cite{BLL} applies;
defining ${\mathbf f^\star} = (f_\alpha^\star: \alpha\in\{0,1\}^k))$,
\[ \scriptt_k({\mathbf f}) \le  \scriptt_k({\mathbf f^\star})\]
for arbitrary ordered $2^k$--tuples ${\mathbf f}$ of nonnegative Lebesgue measurable functions.

Let $f:\reals^d\to[0,\infty]$ be Lebesgue measurable, and suppose that 
$|\{x: f(x)>t\}|<\infty$ for all $t>0$ and that $\norm{f^\star}_{U_k}<\infty$.
Any such function may be represented as
\begin{equation}\label{eq:superp} f = \int_0^\infty \one_{E_t}\,dt \text{ almost everywhere.}\end{equation}

Suppose that $\norm{f}_{U_k} = \norm{f^\star}_{U_k}$.
Applying the representation \eqref{eq:superp} to each copy of $f$ appearing in $\scriptt_k(f,f,\dots,f)$ 
and invoking multilinearity yields
\begin{align*} \norm{f}_{U_k}^{2^k} = \scriptt_k(f,f,\dots,f)
= \int_0^\infty \cdots\int_0^\infty \scriptt_k(\{\one_{E_{t_\alpha}}: \alpha\in\{0,1\}^k\})\,dt \end{align*}
where $\reals^{2^k}\owns t = (t_\alpha: \alpha\in\{0,1\}^k)\in(0,\infty)^{2^k}$.
Likewise,
\begin{align*} \norm{f^\star}_{U_k}^{2^k} = \scriptt_k(f,f,\dots,f)
= \int_0^\infty \cdots\int_0^\infty \scriptt_k(\{\one_{E_{t_\alpha}^\star}: \alpha\in\{0,1\}^k\})\,dt. \end{align*}
Equality of the $U_k$ norms forces
\begin{equation}\label{eq:multimulti} 
\int_0^\infty \cdots\int_0^\infty \scriptt_k(\{\one_{E_{t_\alpha}}: \alpha\in\{0,1\}^k\})\,dt
= \int_0^\infty \cdots\int_0^\infty \scriptt_k(\{\one_{E_{t_\alpha}^\star}: \alpha\in\{0,1\}^k\})\,dt.\end{equation}
By the Brascamp-Lieb-Luttinger inequality,
\begin{equation*}
\scriptt_k(\{\one_{E_{t_\alpha}}: \alpha\in\{0,1\}^k\})
\le
\scriptt_k(\{\one_{E_{t_\alpha}^\star}: \alpha\in\{0,1\}^k\})
\ \ \text{for all $t\in(0,\infty)^{\{0,1\}^k}$.}
\end{equation*}
Since $\norm{f^\star}_{U_k}<\infty$, the multiple integrals in \eqref{eq:multimulti} are finite,
so equality of these integrals together with the pointwise inequality between the integrands
forces
\begin{equation}
\scriptt_k(\{\one_{E_{t_\alpha}}: \alpha\in\{0,1\}^k\})
= \scriptt_k(\{\one_{E_{t_\alpha}^\star}: \alpha\in\{0,1\}^k\})
\text{ for almost every $t\in(0,\infty)^{2^k}$.}
\end{equation}

The function $(0,\infty)\owns t\mapsto E_t$ is right continuous in the sense that
$|E_s \symdif E_t|\to 0$ as $s$ approaches $t$ from above.
This holds because $|E_t|<\infty$ and 
\[ E_t = \{x: f(x)>t\} = \cup_{s>t}\{x: f(x)>s\} = \cup_{s>t} E_s.\]
Therefore by multilinearity together with the
inequality $\scriptt_k(f) \le C\prod_\alpha \norm{f}_{p_k}$,
the function $(0,\infty)^{\{0,1\}^k}\owns t\mapsto
\scriptt_k(\{\one_{E_{t_\alpha}}: \alpha\in\{0,1\}^k\})$
is right continuous in the sense that
\[ \scriptt_k(\one_{E_{s_\alpha}}: \alpha\in\{0,1\}^k)
\to
\scriptt_k(\one_{E_{t_\alpha}}: \alpha\in\{0,1\}^k)
\ \ \text{as $s_\alpha\to t_\alpha^+$ for all $\alpha\in\s01^k$.}  \]

Of course, the function
$(0,\infty)^{2^k}\owns t\mapsto \scriptt_k(\{\one_{E_{t_\alpha}^\star}: \alpha\in\{0,1\}^k\})$
is continuous.
If two functions are right continuous and agree almost everywhere, then they agree everywhere.
Specializing to points $t=(s,s,s,\dots)$, we conclude that
\begin{equation} \norm{{E_s}}_{U_k} = \norm{{E_s^\star}}_{U_k} \end{equation}
for every $s\in(0,\infty)$.
From Theorem~\ref{thm:1} we conclude that
every set $E_s$ is an ellipsoid $\scripte_s$, in the sense that $|E_s\symdif\scripte_s|=0$.
The next lemma therefore suffices to complete the proof of the corollary.
\end{proof}

In the next statement, the empty set is considered to be a ball centered at $0$, and
hence an ellipsoid.
\begin{lemma} \label{lemma:painfulend}
Let $k\ge 2$ and $d\ge 1$.
Suppose that for each $s>0$,
$E_s\subset \reals^d$ is an ellipsoid. Suppose that these sets are nested
in the sense that $s\le t\Rightarrow E_s\supset E_t$. Suppose that
\[\scriptt_k(\one_{E_{t_\alpha}}: \alpha\in\{0,1\}^k)
= \scriptt_k(\one_{E_{t_\alpha}^\star}: \alpha\in\{0,1\}^k)
\ \ \text{for every $t \in (0,\infty)^{\{0,1\}^k}$.} \]
Then there exists an affine automorphism $\phi$ of $\reals^d$ such that for every $s>0$,
$\phi(E_s)$ is a ball centered at $0$. 
\end{lemma}

Lemma~\ref{lemma:painfulend} will be an almost immediate consequence
of the next lemma.
We continue to regard two Lebesgue measurable subsets of $\reals^d$
as identical if their symmetric difference has Lebesgue measure equal to zero.

\begin{lemma} \label{lemma:diagonalminusone}
Let $B\subset\reals^d$ be a closed ball of positive, finite radius centered at $0$.
If $S\subset\reals^d$ satisfies $|S|<|B|$ and 
$\scriptt_k(S,B,\dots,B) = \scriptt_k(S^\star,B,\dots,B)$
then $S$ must be a ball centered at $0$.
\end{lemma}
The notation means that on the left-hand side we have $\scriptt({\mathbf E})$
with $E_\alpha=S$ if $\alpha=(0,0,\dots,0)$ and $E_\alpha=B$ for all other $\alpha$.

The multilinear form $\scriptt_k$ can be expressed as
\begin{equation}
\scriptt_k({\boldf}) 
= \int_{\reals^d} \int_{(\reals^d)^{k}} 
\prod_{\alpha\in\s01^k} f_\alpha(x_0+\alpha\cdot\boldx)\,d\boldx\,dx_0
\end{equation}
where $\boldf=(f_\alpha: \alpha\in\s01^k)$,
$\alpha = (\alpha_1,\dots,\alpha_k)\in\s01^k$, $\boldx=(x_1,\dots,x_k)\in(\reals^d)^k$,
and $\alpha\cdot\boldx = \sum_{j=1}^k \alpha_k x_k\in \reals^d$.
To begin the proof of Lemma~\ref{lemma:diagonalminusone},
consider the set 
\begin{equation} \label{slicee}
K=\set{(x_0,\boldx)\in\reals^d\times (\reals^d)^k: x_0+\alpha\cdot\boldx\in B\ 
\text{ for all $0\ne \alpha\in\s01^k$} }
\end{equation}
and the associated slices $K_y$. 
Since $B$ is compact and convex, $K$ is likewise compact and convex.
It has nonempty interior, so has positive Lebesgue measure.

Define
\begin{equation} \label{Lformula} L(y) = |K_y| =  \int_{(\reals^d)^{k}} 
\prod_{0\ne \alpha\in\s01^k} \one_B(y+\alpha\cdot\boldx)\ d\boldx.  \end{equation}
With $E_\alpha=S$ for $\alpha=0$ and $E_\alpha=B$
for all $\alpha\ne 0$ one has the representation
\begin{equation} \label{roleofL} \scriptt_k(S,B,B,\dots,B) = \int_S L(y)\,dy.\end{equation}
$K$ is invariant under the diagonal action of the orthogonal group $O(d)$
on $\reals^d\times (\reals^d)^k$, so
$K$ is balanced and $L$ is a radially symmetric nonnegative function. 

The following property of $L$ will be useful in the proof.
\begin{sublemma}\label{sublemma}
Let $B\subset\reals^d$ be a ball of finite radius centered at $0$. 
Then $L(y)$ is a strictly decreasing function of $|y|$ in $B$.
\end{sublemma}

\begin{proof}[Proof of Sublemma~\ref{sublemma}]
Let $\rho$ continue to denote the radius of $B$, and
let $y\in\reals^d$ satisfy $|y|<\rho$.
Define $\boldx = -k^{-1}(y,y,\dots,y)$. 
For any $\alpha\in\s01^k$, 
$ y + \alpha\cdot\boldx =  (1- k^{-1} \sum_{j=1}^k \alpha_j)y$.
Since $1-k^{-1}\sum_j \alpha_j\in[0,1]$,
$|y+\alpha\cdot\boldx| \le |y|<\rho$.
Thus $(y,\boldx)$ lies in the interior of $K_0$.
Consequently $L(y)>0$.

If there exists $y\ne 0$ for which $|K_y|=|K_0|$ then
according to Lemma~\ref{lemma:convexslices},
there exists $\boldw\in\reals^d$ such that $K_y=K_0+\boldw$.
This equality is equivalent to
\[ |\alpha\cdot\boldx| \le \rho \ \forall\,\alpha\ne 0
\ \Leftrightarrow\ 
|y+\alpha\cdot\boldw+\alpha\cdot\boldx|\le \rho  \ \forall\, \alpha\ne 0.\]

Suppose that there exists $0\ne \beta\in\s01^k$ such that $y+\beta\cdot\boldw\ne 0$. 
Express $y+\beta\cdot\boldw = tv$ where $v\in\reals^d$ is a unit vector
and $t>0$.
Let $n$ be the number of indices $j$ for which $\beta_j=1$.
Define $\boldx\in (\reals^d)^k$ by setting
$x_j=0$ if $\beta_j=0$, and $x_j = n^{-1}\rho v$ if $\beta_j=1$.
Thus $\beta\cdot\boldx=\rho v$.
For any $\alpha\in\s01^k$
\[|\alpha\cdot\boldx| \le \sum_j \alpha_j |x_j|
= \sum_{j: \beta_j=1} \alpha_j n^{-1}\rho \le  \sum_{j: \beta_j=1} n^{-1}\rho = \rho.\]
Thus $\boldx\in K_0$.
On the other hand,
\[ |y+\beta\cdot\boldw + \beta\cdot\boldx|
= |y+\beta\cdot\boldw + \rho| 
= tv + \rho v
=t+\rho
>\rho,\]
so $y+\beta\cdot \boldw + \boldx\notin K_0$.
So $K_0+\boldw\ne K_y$.

The only remaining possibility is that
$y+\alpha\cdot\boldw=0$ for all nonzero $\alpha\in\s01^k$. 
This implies that $y=0$.
\end{proof}

\begin{proof}[Completion of proof of Lemma~\ref{lemma:diagonalminusone}]
Recall the representation \eqref{roleofL} for $\scriptt_k(S,B,B,\dots,B)$.
Since $L$ is radially symmetric and nonincreasing,
$\int_S L \le \int_{S^\star} L$.
$B$ contains $S^\star$ since both are balls centered at $0$
and $|S|\le |B|$.
According to Sublemma~\ref{sublemma}, $L(y)$
is a strictly decreasing function of $|y|$ in a ball that contains $S^\star$.
Therefore $\int_S L$ can only equal $\int_{S^\star}L$
if $|S\symdif S^\star|=0$.
\end{proof}

\begin{proof}[Proof of Lemma~\ref{lemma:painfulend}]
The conclusion of the lemma is equivalent to stating
that for any $s,t\in(0,\infty)$,
the ellipsoids $E_s,E_t$ are mutually concentric and homothetic.
We may suppose without loss of generality that $s>t$,
ensuring that $E_s\subset E_t$ and hence $|E_s|\le |E_t|$.
We may also assume that $|E_s|>0$.
For any Lebesgue measure-preserving affine automorphism $\phi$ of $\reals^d$,
\[\scriptt_k(\phi(E_s),\phi(E_t),\phi(E_t),\dots,\phi(E_t)) 
= \scriptt_k(E_s,E_t,E_t,\dots,E_t) 
= \scriptt_k(E_s^\star,E_t^\star,E_t^\star,\dots,E_t^\star).\]
Choose $\phi$ so that $\phi(E_t)$ is a ball centered at $0$
and invoke Lemma~\ref{lemma:diagonalminusone} to conclude that $\phi(E_s)$ is likewise
a ball centered at $0$, establishing the claim.
\end{proof}

\end{document}